\documentclass[12pt,a4paper]{article}
\usepackage[left=3cm,right=3cm,top=1.5cm,bottom=1.5cm]{geometry}
\usepackage[utf8]{inputenc}
\usepackage[english]{babel}
\usepackage[T1]{fontenc}
\usepackage{amsmath}
\usepackage{amsthm}
\usepackage{amsfonts}
\usepackage{amssymb}
\usepackage{dsfont}
\usepackage{graphicx}
\usepackage[usenames,dvipsnames,svgnames,table]{xcolor}
\usepackage{kpfonts}
\usepackage{extarrows}
\usepackage{natbib}
\usepackage[colorlinks,citecolor=blue,linkcolor=blue]{hyperref}
\usepackage{enumitem}

\newcommand{\Z}{\mathbb{Z}}

\newcommand{\R}{\mathbb{R}}

\newcommand{\ESP}{\mathbb{E}}

\newcommand{\sas}{\mathcal{S}\alpha\mathcal{S}}
\newcommand{\Za}[1]{\mathrm{Z}_{\alpha} \left( {#1} \right) }
\newcommand{\supp}[1]{\mathrm{supp} \left( {#1} \right)}

\def\ce{{\cal C}}

\def\La{\mathrm{L}^{\alpha}}

\def\o{\omega}

\def\al{\alpha}

\def\ind{\mathds{1}}

\newtheorem{thm}{Theorem}[section]
\newtheorem{lemma}{Lemma}[section]
\newtheorem{proposition}{Proposition}[section]
\newtheorem{corollary}{Corollary}[section]
\newtheorem{definition}{Definition}[section]
\newtheorem{rem}{\textbf{Remark}}[section]

\title{Behaviour of linear multifractional stable motion: membership of a critical H\"older space}
\author{Antoine Ayache \\
Univ. Lille, CNRS, UMR 8524 - Laboratoire Paul Painlevé,\\ F-59000 Lille, France.\\
E-mail: \texttt{Antoine.Ayache@math.univ-lille1.fr}\\
\and
Julien Hamonier \\
Univ. Lille, CHU Lille, EA 2694 - Santé publique : \\ épidémiologie et qualité des soins, F-59000 Lille, France.\\
E-mail: \texttt{Julien.Hamonier@univ-lille2.fr}
 }
\date{}

\begin{document}
\maketitle

\begin{abstract}
The study of path behaviour of stochastic processes is a classical topic in probability theory and related areas. In this frame, a natural question one can address is: whether or not sample paths belong to a critical H\"older space? The answer to this question is negative in the case of Brownian motion and many other stochastic processes: it is well-known that despite the fact that Brownian paths satisfy, on each compact interval $I$, a H\"older condition of any order strictly less than $1/2$, they fail to belong to the critical H\"older space $\ce^{1/2}(I)$. In this article, we show that a different phenomenon happens in the case of linear multifractional stable motion (LMSM): for any given compact interval one can find a critical H\"older space to which sample paths belong. Among other things, this result improves an upper estimate, recently derived in \citep{Bierme13}, on behaviour of LMSM, by showing that the logarithmic factor in it is not needed.
\end{abstract}



\section{Introduction and statement of the main result}

Over the two last decades, there has been a growing interest in probabilistic models based on fractional and multifractional processes (see e.g. \citep{ayache2005multifractional,ayache2007wavelet,
ayache2011multiparameter,falconer2002tangent,falconer2003local,roux1997elliptic,peltier1995multifractional,
lacaux04,Sur,Bi,BiP,BiPP,ayache2011prostate,stoev2004stochastic,stoev2005path,meerschaert2008local,bardet2010nonparametric,dozzi2011,HLS12,LRMT11,stoev2006rich,hamonier2012lmsm,Bierme13,balancca2014fine,balancca2015mul}). Actually, they convey a convenient framework for modelling in several fields such as Internet traffic and finance. 

In this article, we focus on one of the the most typical multifractional processes with heavy-tailed stable distributions: the linear multifractional stable motion (LMSM), which was first introduced in \citep{stoev2004stochastic,stoev2005path} and studied in \citep{hamonier2012lmsm,Bierme13,falconer2009localizable,falconer2009multifractional,balancca2014fine}. In order to precisely define LMSM, let $\al\in (1;2)$ be a fixed real number and $\Za{ds}$ be an independently scattered symmetric $\al$-stable ($\sas$) random measure on $\R$, with Lebesgue control measure \citep{SamTaq}; the underlying probability space is denoted by $(\Omega,\mathcal{F},\mathbb{P})$. Also, let $H(\cdot)$ be an arbitrary deterministic continuous function defined on the real line with values in $(1/\al;1)$.

\begin{definition}
\label{defj:lmsm}
The LMSM $\{Y(t):t\in \R\}$ with Hurst functional parameter $H(\cdot)$, is defined, for each $t\in \R$, as,
\begin{equation}\label{m:def:lmsm}
Y(t)=X(t,H(t)),
\end{equation}
where $X:=\big\{X(u,v):(u,v)\in \R\times (1/\al;1)\big\}$ is the real-valued $\sas$ random field on the probability space $(\Omega,\mathcal{F},\mathbb{P})$, such that, for every $(u,v)\in \R\times (1/\al;1)$,
\begin{equation}
\label{def:champX}
X(u,v)=\int_{\R}\Big\{ (u-s)_+^{v-1/\alpha} - (-s)_+^{v-1/\alpha} \Big\} \Za{ds}.
\end{equation}
Recall that for all $(x,\kappa)\in\R^2$, 
\begin{equation}
\label{h:eq:pospart}
(x)_+^{\kappa}=x^{\kappa}\mbox{ if $x>0$ and}\,\,(x)_+^{\kappa}=0\mbox{ else.}
\end{equation}
\end{definition}

\begin{rem}
Notice that in the special case where $H(\cdot)$ is a constant function denoted by $H$, LMSM reduces to linear fractional stable motion (LFSM) of Hurst parameter $H$; we mention that sample path properties of the latter process have already been studied in \citep{Tak89,maejima1983self,kono1991holder,SamTaq,balancca2014fine}. Also notice that when $\alpha=2$ LMSM reduces to the classical multifractional Brownian motion (MBM) of functional parameter $H(\cdot)$ which was introduced in the mid 1990s independently by \citep{roux1997elliptic} and by \citep{peltier1995multifractional}.
\end{rem}

In order to precisely explain the motivation behind the present article, one has to make a brief presentation of some previous results on path behaviour of LMSM which were obtained in the literature. For the sake of clarity, first, it is useful to recall the precise definition of the H\"older space $\ce^{\gamma}(I;\R)$.
\begin{definition}\label{def:holderspace}
Let $I$ be a compact interval of $\R$.  For each real number $\gamma \in [0;1)$, the H\"older space $\ce^{\gamma}(I):=\ce^{\gamma}(I;\R)$ is the set of the real-valued continuous functions $f$ on $I$ satisfying:
$$
\sup_{(x,y)\in I^2;\, x\neq y}\left\{\frac{|f(x)-f(y)|}{|x-y|^{\gamma}}\right\}<\infty.
$$
\end{definition}

Stoev and Taqqu obtained, in their article \citep{stoev2005path}, a first result about global path behaviour of LMSM on a compact interval. Namely, thanks to a strong version of the Kolmogorov continuity theorem, they showed the following result.
\begin{thm}\citep[]{stoev2005path}\label{thm:stoevtaqquH}
Let $I$ be a compact interval of $\R$. Assume that the functional parameter $H(\cdot)$ of LMSM $Y$ satisfies:
$$
(\mathbf{A}) \qquad H(I) \subset (1/\alpha;1) \mbox{ and } H(\cdot) \in \ce^{\gamma_H}(I) \mbox{ for some $\gamma_H\in (1/\al;1)$.}
$$
Then LMSM has a modification whose paths are, with probability $1$, continuous functions on $I$. Moreover, for all $\gamma < \min_{t\in I} H(t)- 1/\alpha$, they belong to the H\"older space $\ce^{\gamma}(I)$.
\end{thm}
\noindent

Later, thanks to a different methodology, relying on wavelet series representations of the field $X$ and the corresponding LMSM $Y$, Theorem~\ref{thm:stoevtaqquH} has been improved in the article \citep{hamonier2012lmsm} in the following way: under condition $(\mathbf{A})$, for any $\eta>0$, one has

\begin{equation}
\label{eq3:modcont2Y}
\sup_{(t,s)\in I^2;\, t\neq s} \left\{\frac{\big|Y(t)-Y(s)\big|}{ |t-s|^{\min_{t\in I} H(t)-1/\alpha} \left(1+\big|\log |t-s| \big| \right)^{2/\alpha+\eta}}\right\} < \infty \qquad a.s.
\end{equation}

The article \citep{hamonier2012lmsm} has also shown that, under condition $(\mathbf{A})$, the upper estimate provided by (\ref{eq3:modcont2Y}) on sample paths behaviour of LMSM is quasi-optimal in the following sense: for any $\eta>0$, one has
\begin{equation}
\label{eq4:modcont2Y}
\sup_{(t,s)\in I^2;\, t\neq s}\left\{\frac{\big|Y(t)-Y(s)\big|}{ |t-s|^{\min_{t\in I} H(t)-1/\alpha} \left(1+\big|\log |t-s| \big| \right)^{-\eta}}\right\} = \infty \qquad a.s.
\end{equation}

Before finishing this brief presentation of previous results on path behaviour of LMSM, we mention that in their article \citep{Bierme13}, Bierm\'e and Lacaux have  improved, under condition $(\mathbf{A})$, the upper estimate provided by (\ref{eq3:modcont2Y}) on sample paths behaviour of LMSM, in the following way: one has 
\begin{equation}\label{eq5:modcont2Y}
\sup_{(t,s)\in I^2;\, t\neq s} \left\{\frac{\big|Y(t)-Y(s)\big|}{ |t-s|^{\min_{t\in I} H(t)-1/\alpha} \sqrt{1+|\log |t-s||}}\right\}  < \infty \qquad a.s.
\end{equation}

As a conclusion, in view of Relations~(\ref{eq3:modcont2Y}), (\ref{eq4:modcont2Y}) and (\ref{eq5:modcont2Y}), it is natural to address the following question: {\em  assuming that condition $(\mathbf{A})$ holds, do LMSM's sample paths belong to the critical H\"older space $\ce^{\min_{t\in I} H(t)- 1/\alpha}(I)$?} 

At first sight, one is tempted to believe that the answer to such a question is negative, since this is the case for many other examples of stochastic processes. For instance, it is well-known that despite the fact that Brownian paths satisfy, on each compact interval $I$, a H\"older condition of any order strictly less than $1/2$, they fail to belong to the critical H\"older space $\ce^{1/2}(I)$. 

Yet, we will show that for LMSM the answer to the question is positive (see Corollary~\ref{main:cor} given below). In fact, this will be a consequence of our main result that we are now going to state. 


\begin{thm}\label{main:result}
We denote by $I$ be a compact interval of $\R$, and we assume that the functional parameter $H(\cdot)$ of LMSM $Y$ satisfies:
$$
(\mathbf{A'}) \qquad H(I) \subset (1/\alpha;1) \mbox{ and } H(\cdot) \in \ce^{\gamma_H}(I) \mbox{ for some $\gamma_H\in [0,1)$.}
$$
Notice that condition $(\mathbf{A'})$ is more general than condition $(\mathbf{A})$.
Let $\Omega_0^*$ and $\Omega_1^*$ be the two events of probability~$1$ defined respectively in Lemma~\ref{omega0} and Proposition~\ref{omega1} given in the next section (it is clear that their intersection $\Omega_0^* \cap \Omega_1^*$ is also an event probability~$1$). For any $\omega\in\Omega_0^* \cap \Omega_1^*$, one has
\begin{equation}\label{eq1:main:result}
\sup_{(t,s)\in I^2;\, t\neq s} \left\{ \frac{|Y(t,\omega)-Y(s,\omega)|}{|t-s|^{\gamma_H \wedge \left(\min_{x\in I} H(x) -1/\al\right)}}    \right\} < \infty.
\end{equation}
\end{thm}
Observe that, when the functional parameter $H(\cdot)$ of LMSM $Y$ satisfies condition~$(\mathbf{A})$, one can derive from Relation~(\ref{eq1:main:result}) and the inequalities $\gamma_H>1/\al>\min_{x\in I} H(x) -1/\al$ the following result.

\begin{corollary}
\label{main:cor}
Let $I$ be a compact interval of $\R$. We assume that Condition~$(\mathbf{A})$ holds for this interval. Then, the sample paths of LMSM $Y$ almost surely belong to the critical H\"older space $\ce^{\min_{t\in I} H(t) -1/\al}(I)$. 
\end{corollary} 

\begin{rem}
In the special case of LFSM (i.e. $H(\cdot)$ is a constant function denoted by $H$), Corollary~\ref{main:cor} has already been obtained by Takashima in \citep{Tak89}. 
\end{rem}


The rest of the article is organized in the following way. In Section~\ref{sect:metho}, first one makes some recalls related with the wavelet series representation of the field $X$, since it provides the main tool for proving Theorem~\ref{main:result}. Then one describes the strategy employed in order to get this theorem, and one states the intermediate results which are needed. The  proofs of the new results among them are given in Section~\ref{sect:proofs}.

\section{Methodology}\label{sect:metho}

%
Let $X:=\big\{X(u,v):(u,v)\in \R\times (1/\al;1)\big\}$ be the real-valued $\sas$ random field which has been defined in (\ref{def:champX}) and has been associated to LMSM $Y$ through~(\ref{m:def:lmsm}). The main tool for proving Theorem~\ref{main:result} is the wavelet series representation of $X$ which was introduced  in \citep{hamonier2012lmsm}. First, we will make some brief recalls related with this representation of $X$. To this end, we denote by  $\psi:\R\rightarrow\R$ a 3 times continuously differentiable compactly supported Daubechies mother wavelet which generates an orthonormal basis of $L^2 (\R)$. We mention in passing that three very classical references on the wavelet theory are the books \citep{Dau92,Meyer90,Meyer92}. 

The following theorem, which has been obtained in \citep{hamonier2012lmsm}, provides the wavelet series representation of $X$.

\begin{thm}\citep{hamonier2012lmsm}
\label{thm:wavrepX}.
Let $\Psi:\R\times (1/\al;1)\rightarrow \R$ be the function defined, for all $(u,v)\in \R \times (1/\alpha;1)$, as
\begin{equation*}\label{daub:PSI}
\Psi(x,v):=\int_{\R} (x-s)_{+}^{v-1/\alpha} \psi(s) ds\,.
\end{equation*}
Let $\lbrace \epsilon_{j,k} : (j,k)\in\Z^2 \rbrace$ be the sequence of the identically distributed $\sas$ real-valued random variables defined, for every $(j,k)\in\Z^2$, 
as
\begin{equation}
\label{daubdef:ejk}
\epsilon_{j,k} := 2^{j/\alpha}\int_{\R}  \psi(2^js-k) \Za{ds}.
\end{equation}
Then, for any $(u,v)\in \R \times (1/\alpha;1)$, the random variable $X(u,v)$ can be, almost surely, expressed as
\begin{equation}
\label{eq1:descrip}
X(u,v)=\sum_{(j,k)\in \Z^2} 2^{-jv}\epsilon_{j,k}\big(\Psi (2^j u-k,v)-\Psi (-k,v)\big), 
\end{equation}
where the series is on $\Omega_0^*$, the event of probability $1$ introduced in Lemma~\ref{omega0} stated below, uniformly convergent in $(u,v)$, on any compact subset of $\R \times (1/\alpha;1)$.
\end{thm}

\noindent From now on the field $X$ is identified with its modification provided by (\ref{eq1:descrip}). We recall that the following two lemmas, which have respectively been derived in \cite{hamonier2012lmsm} and \cite{ayache2009linear}, are the two main ingredients of the proof of Theorem~\ref{thm:wavrepX}.

\begin{lemma}{\citep[]{hamonier2012lmsm}}\label{localisation}  
For all $(p,q) \in \{ 0,1,2,3 \}\times\Z_+$, the partial derivative $\partial_x^p \partial_v^q \Psi$ exists and is a continuous function on $\R \times (1/\alpha;1)$, with the usual convention that $\partial_x^0 \partial_v^0 \Psi:=\Psi$. Moreover, all of these functions ($\Psi$ and its existing partial derivatives) are well-localized in $x$ uniformly $v\in [a;b]$, that is one has: 
\begin{equation}
\label{eq:localisation}
\forall\, (p,q) \in \{ 0,1,2,3 \}\times\Z_+, \quad \sup_{(x,v)\in\R\times [a;b]}\Big\{ (3+|x|)^2 \big|(\partial_x^p \partial_v^q \Psi)(x,v)\big|\Big\}< \infty,
\end{equation}
where $[a;b]$ denotes an arbitrary nonempty compact subinterval of $(1/\alpha;1)$.
\end{lemma}

\begin{lemma}{\citep[]{ayache2009linear}}\label{omega0}
There exists an event of probability $1$, denoted by $\Omega_0^*$, such that, for every fixed real number $\eta>0$, one has, for all $\omega\in\Omega_0^*$ and for each $(j,k)\in\Z^2$,
\begin{equation}
\label{eq1:omega0}
\big|\epsilon_{j,k}(\o)\big| \le C(\o)\big(3+|j| \big)^{1/\alpha+\eta}\big(3+|k| \big)^{1/\alpha+\eta}, 
\end{equation}
where $C$ is a positive and finite random variable only depending on $\eta$.
\end{lemma}

\begin{rem}
\label{rem:contX}
In view of Lemma~\ref{localisation} and of the uniform convergence property of series in (\ref{eq1:descrip}), for each fixed $\o\in\Omega_0^*$, the sample path 
$(u,v)\mapsto X(u,v,\o)$ is a continuous function on $\R\times (1/\al;1)$. 
\end{rem}

Before  finishing our recalls, let us give a useful proposition which was obtained in \citep{{hamonier2012lmsm}} thanks to Lemmas~\ref{localisation} and \ref{omega0}.
\begin{proposition}\citep[]{hamonier2012lmsm}
\label{cor:TWSEbis}
For each fixed $\o\in \Omega_{0}^*$ and real numbers $M,a,b$ satisfying $M>0$ and $1/\al <a<b<1$, one has (with the convention that $0/0=0$)
\begin{equation}
\label{eq2:CTWSEbis}
\sup_{(u,v_1,v_2) \in [-M;M]\times [a;b]^2} \left\{\frac{\big| X(u,v_1,\omega) - X(u,v_2,\omega)\big|}{ |v_1-v_2|}\right\}  < \infty.
\end{equation}
\end{proposition}

Having finished our recalls related with the wavelet series representation of $X$, let us now turns to description of the strategy allowing to derive Theorem~\ref{main:result}. First, we point out that this theorem is mainly a consequence of the following more general result.
\begin{thm}\label{cor:holder}
Let $\Omega_0^*$ be the event of probability~$1$ which has already been introduced in Lemma~\ref{omega0}. Let $\Omega_1^*$ be the event of probability~$1$ which will be introduced in Proposition~\ref{omega1} given below. 
For all fixed $\omega\in \Omega_0^* \cap \Omega_1^*$ and real numbers $M,a,b$ satisfying $M>0$ and $1/\alpha<a<b<1$, one has (with the convention that $0/0=0$)
\begin{equation}\label{eq:cor:holder}
\sup_{(u_1,u_2,v_1,v_2)\in [-M;M]^2\times [a;b]^2}\left\lbrace \frac{|X(u_1,v_1,\omega)-X(u_2,v_2,\omega)|}{|u_1-u_2|^{v_1\vee v_2 -1/\alpha} + |v_1-v_2|} \right\rbrace <+\infty\,,
\end{equation}
where $v_1\vee v_2:=\max (v_1,v_2)$.
\end{thm}
\noindent We mention that, even though it is more general than Theorem~\ref{main:result}, we have prefered to state Theorem~\ref{cor:holder} in the present section, and not in the introductory section, since it is a bit technical result. Let us now explain how Theorem~\ref{cor:holder} allows to obtain Theorem~\ref{main:result}.

\begin{proof}[Proof of Theorem \ref{main:result}]
There is no restriction to assume that the compact interval $I$ is of the form $I=[-M;M]$, where $M$ is some positive real number. Let us then consider an arbitrary $\o\in\Omega_0^* \cap \Omega_1^*$ and an arbitrary couple $(t,s) \in [-M;M]^2$, such that $t\neq s$. Using Relation~(\ref{m:def:lmsm}), Theorem~\ref{cor:holder}, condition $(\mathbf{A'})$, and Definition~\ref{def:holderspace}, one gets that
\begin{align*}
\big|Y(t,\o)-Y(s,\o)\big| & = \big|X(t,H(t),\o)-X(s,H(s),\o)\big|  \\
		&\le C(\o)\left( |t-s|^{H(t)\vee H(s)-1/\alpha}+|H(t)-H(s)| \right) \\
		& \le C(\o) |t-s|^{H(t)\vee H(s)-1/\alpha}+C'(\o)|t-s|^{\gamma_H}\\
		& \le C''(\o)|t-s|^{\min_{x\in I} H(x)-1/\alpha}+C'(\o)|t-s|^{\gamma_H}\\
		& \le C'''(\o)|t-s|^{\gamma_H \wedge \left( \min_{x\in I} H(x)-1/\alpha\right)}\,,
\end{align*}
where $C,\ldots, C'''$ are positive and finite random variables not depending on $(t,s)$. 
\end{proof}

Theorem~\ref{cor:holder} can be derived from Proposition~\ref{cor:TWSEbis} and the following proposition.
\begin{proposition}\label{main:prop}
For all fixed $\omega\in \Omega_0^* \cap \Omega_1^*$ and real numbers $M,a,b$ satisfying $M>0$ and $1/\alpha<a<b<1$, one has  (with the convention that $0/0=0$)
\begin{equation*}\label{eqn:regul3}
\sup_{(u_1,u_2,v)\in [-M;M]^2\times [a;b]}\left\lbrace \frac{|X(u_1,v,\omega)-X(u_2,v,\omega)|}{|u_1-u_2|^{v-1/\alpha}} \right\rbrace <\infty.
\end{equation*}
\end{proposition}


\begin{proof}[Proof of Theorem~\ref{cor:holder}]
For all $\o\in\Omega_0^* \cap \Omega_1^*$ and for each $(u_1,u_2,v_1,v_2) \in [-M;M]^2\times [a;b]^2$, one sets
$$
f(u_1,u_2,v_1,v_2,\o):=\frac{\big| X(u_1,v_1,\o) - X(u_2,v_2,\o)\big|}{ |u_1-u_2|^{v_1\vee v_2-1/\alpha} +|v_1-v_2|}\,,
$$
with the convention that $0/0=0$. Next setting $v_1\wedge v_2:=\min (v_1,v_2)$, and using the fact that 
\[
f(u_1,u_2,v_1,v_2,\o)=f(u_2,u_1,v_2,v_1,\o)\,,
\]
 it follows that
\begin{align}\label{eq2:modcont}
& \sup_{(u_1,u_2,v_1,v_2) \in [-M;M]^2\times [a;b]^2} \Bigg\{\frac{\big| X(u_1,v_1,\o) - X(u_2,v_2,\o)\big|}{ |u_1-u_2|^{v_1\vee v_2-1/\alpha}+|v_1-v_2|} \Bigg\} \nonumber \\
&  = \sup_{(u_1,u_2,v_1,v_2) \in [-M;M]^2\times [a;b]^2} \Bigg\{ \frac{\big| X(u_1,v_1\vee v_2,\o) - X(u_2,v_1\wedge v_2,\o)\big|}{ |u_1-u_2|^{v_1\vee v_2-1/\alpha} +|v_1-v_2|} \Bigg\}.
\end{align}
Moreover, using the triangle inequality, and the inequality, for all $(u_1,u_2,v_1,v_2) \in [-M;M]^2\times [a;b]^2$,
\begin{align*}
\max\left\{|u_1-u_2|^{v_1\vee v_2-1/\alpha}, |v_1-v_2|\right\} \le |u_1-u_2|^{v_1\vee v_2-1/\alpha} +|v_1-v_2|\,,
\end{align*}
one gets that
\begin{align}
\label{eq3:modcont}
& \sup_{(u_1,u_2,v_1,v_2) \in [-M;M]^2\times [a;b]^2} \left\{\frac{\big| X(u_1,v_1\vee v_2,\o) - X(u_2,v_1\wedge v_2,\o)\big|}{ |u_1-u_2|^{v_1\vee v_2-1/\alpha} +|v_1-v_2|} \right\} \nonumber \\
& \le \sup_{(u_1,u_2,v_1,v_2) \in [-M;M]^2\times [a;b]^2} \left\{\frac{\big| X(u_1,v_1\vee v_2,\o) - X(u_2,v_1\vee v_2,\o)\big|}{ |u_1-u_2|^{v_1\vee v_2-1/\alpha}+|v_1-v_2|} \right\} \nonumber \\
& \qquad +\sup_{(u_1,u_2,v_1,v_2) \in [-M;M]^2\times [a;b]^2} \left\{\frac{\big| X(u_2,v_1\vee v_2,\o) - X(u_2,v_1\wedge v_2,\o)\big|}{ |u_1-u_2|^{v_1\vee v_2-1/\alpha} +|v_1-v_2|} \right\} \nonumber \\
& \le \sup_{(u_1,u_2,v) \in [-M;M]^2\times [a;b]} \left\{\frac{\big| X(u_1,v,\o) - X(u_2,v,\o)\big|}{ |u_1-u_2|^{v-1/\alpha} } \right\} \nonumber\\
& \qquad +\sup_{(u, v_1,v_2) \in [-M;M]\times [a;b]^2} \left\{\frac{\big| X(u,v_1,\o) - X(u, v_2,\o)\big|}{ |v_1-v_2|} \right\}.
\end{align}
Finally, combining (\ref{eq2:modcont}) and (\ref{eq3:modcont}) with Propositions~\ref{cor:TWSEbis}~and~\ref{main:prop} together, one obtains~(\ref{eq:cor:holder}).
\end{proof}

For proving Proposition~\ref{main:prop} one splits the field $X$ in two parts denoted by \\
$\dot{X}:=\big\{\dot{X}(u,v):(u,v)\in \R\times (1/\al;1)\big\}$ and $\ddot{X}:=\big\{\ddot{X}(u,v):(u,v)\in \R\times (1/\al;1)\big\}$. $\dot{X}$ is called the low frequency part, and $\ddot{X}$ the high frequency part. These two $\sas$ random fields are defined, by setting, for all $(u,v)\in \R\times (1/\al;1)$,
\begin{eqnarray}
\label{eq:lfpartX}
\dot{X}(u,v) & =& \sum_{j=-\infty}^{-1} \sum_{k\in\Z} 2^{-jv}  \epsilon_{j,k} \big ( \Psi(2^j u-k,v)-\Psi(-k,v) \big) \nonumber \\
& = & \sum_{j=1}^{+\infty} \sum_{k\in\Z} 2^{jv}  \epsilon_{-j,k} \big (\Psi(2^{-j} u-k,v)-\Psi(-k,v)\big)\,, 
\end{eqnarray}
and
\begin{equation}
\label{eq:hfpartX}
\ddot{X}(u,v) = \sum_{j=0}^{+\infty} \sum_{k\in\Z} 2^{-jv}  \epsilon_{j,k} \big (\Psi(2^j u-k,v)-\Psi(-k,v) \big)\,. 
\end{equation}

\begin{rem}
\label{rem:cont-lowhigh}
One can show that, on the event $\Omega_0^*$, the two series defining $\dot{X}$ and $\ddot{X}$ share exactly the same uniform convergence property with respect to $(u,v)$ as the series in (\ref{eq1:descrip}). Therefore, in view of Lemma~\ref{localisation}, for each fixed  $\omega\in\Omega_0^*$, the sample paths $(u,v)\mapsto \dot{X}(u,v,\omega)$ and $(u,v)\mapsto \ddot{X}(u,v,\omega) $ are continuous functions on $\R\times (1/\alpha;1)$.
\end{rem}
\noindent 
\begin{proof}[Proof of Proposition~\ref{main:prop}]
One clearly has that 
\[
X(u,v,\o)=\dot{X}(u,v,\o)+\ddot{X}(u,v,\o)\,,
\]
for all $(u,v)\in\R\times (1/\alpha;1)$ and $\o\in\Omega_0^*$. Therefore  Proposition~\ref{main:prop} is a straightforward consequence of the following two propositions.
\end{proof}


\begin{proposition}\label{prop:regul1}
For each fixed $\omega\in\Omega_0^*$, the sample path $(u,v) \mapsto\dot{X}(u,v,\omega)$ is a 3 times continuously differentiable function on $\R\times (1/\alpha;1)$. Therefore, it is a Lipschitz function on every compact rectangle $[-M;M]\times [a;b]\subset \R\times (1/\alpha;1)$; in other words, one has that (with the convention that $0/0=0$)
\[
\sup_{(u_1,u_2,v_1,v_2)\in [-M;M]^2\times [a;b]^2}\left\lbrace \frac{|\dot{X}(u_1,v_1,\omega)-\dot{X}(u_2,v_2,\omega)|}{|u_1-u_2|+|v_1-v_2|} \right\rbrace <\infty\,,
\]
and consequently that
\[
\sup_{(u_1,u_2,v)\in [-M;M]^2\times [a;b]}\left\lbrace \frac{|\dot{X}(u_1,v,\omega)-\dot{X}(u_2,v,\omega)|}{|u_1-u_2|} \right\rbrace <\infty\,.
\]
\end{proposition}

\begin{proposition}\label{prop:regul2}
For all fixed $\omega\in \Omega_0^* \cap \Omega_1^*$ and any real numbers $M,a,b$ satisfying $M>0$ and $1/\alpha<a<b<1$, one has (with the convention that $0/0=0$)
\begin{equation}\label{eq1:regul2}
\sup_{(u_1,u_2,v)\in [-M;M]^2\times [a;b]}\left\lbrace \frac{|\ddot{X}(u_1,v,\omega)-\ddot{X}(u_2,v,\omega)|}{|u_1-u_2|^{v-1/\alpha}} \right\rbrace <\infty\,.
\end{equation}
\end{proposition}
\noindent
These two propositions are proved in Section~\ref{sect:proofs}. 

We mention that the strategy of the proof of Proposition~\ref{prop:regul1} is to show that on the 
event $\Omega_0 ^*$, the uniform convergence property with respect to $(u,v)$ of the series in (\ref{eq:lfpartX}) is preserved when the partial derivative operator $\partial_x^p \partial_v^q$, with $(p,q)\in \{0,1,2,3\}\times \Z^+$, is applied to each term of the series.

Let us emphasize that the keystone of the proof of the crucial Proposition~\ref{prop:regul2} consists in sharpening the estimate of $\epsilon_{j,k}$ provided by~(\ref{eq1:omega0}), when $j\ge 0$ and
$|k|2^{-j}$ is bounded by an arbitrary finite constant not depending on $(j,k)$. More precisely the keystone consists in the following proposition.

\begin{proposition}\label{omega1}
There exists an event of probability 1, denoted by $\Omega_1^{*}$, such that, for every fixed positive real number $l$, one has, for all $\omega \in\Omega_1^*$ and for each $(j,k)\in\Z^+\times\Z$ satisfying $|k|2^{-j}\le l$, 
\begin{equation}\label{eq1:omega1}
\big| \epsilon_{j,k}(\omega) \big| \le C'(\o) 2^{j/\alpha}, 
\end{equation}
where $C'$ is positive and finite random variable.
\end{proposition}

\noindent The proof of Proposition~\ref{omega1} will be given in Section~\ref{sect:proofs}. We mention that it relies on the following lemma which can be viewed as a integration by parts formula  between a twice differentiable compactly supported function and a L\'evy $\sas$ process. The proof of this lemma will also be given in Section~\ref{sect:proofs}.
\begin{lemma}\label{prop:ipp}
Let $f:\R\rightarrow\R$ be a twice times continuously differentiable compactly supported function, and let $\Za{\cdot}$ be a $\sas$ random measure with a parameter $\al$ belonging to $(1;2)$. Then, one has, almost surely
\begin{equation*}
\int_{\R}f(s)\Za{ds} = - \int_{\R} f'(s)\Za{s} ds,
\end{equation*}
where $\Za{s}:=\Za{[0;s]}$ if $s\ge 0$ and $\Za{s}:=-\Za{[s;0]}$ otherwise. Notice that the stochastic process $\{\Za{s}:s\in\R\}$ defined in this way is a L\'evy $\sas$ process which we always identify with its c\`adl\`ag modification. 
\end{lemma}

\section{Proofs of intermediate results}\label{sect:proofs}

The proofs of the intermediate results are given in the following natural order: firstly that of Lemma~\ref{prop:ipp}, secondly that of Proposition~\ref{omega1}, thirdly that of  Proposition~\ref{prop:regul2}, and finally that of  Proposition~\ref{prop:regul1}. 

\begin{proof}[Proof of Lemma~\ref{prop:ipp}]
Let $a<b$ be two fixed real numbers such that 
\begin{equation}
\label{eq1:prop:ipp}
\supp{f} \subseteq [a;b]\,.
\end{equation}
 For any fixed integer $j\ge 0$, one denotes by $(y_{j,k})_{0\le k\le 2^j}$ the partition of the interval $[a;b]$ such that 
\begin{equation}
\label{eq2:prop:ipp}
y_{j,k}:=a+k 2^{-j}(b-a)\,, \quad\mbox{for all $k\in\{0,\ldots, 2^j\}$.}
\end{equation}
Moreover, one denotes by $\theta_j$ the step function  defined as 
\begin{equation}
\label{eq3:prop:ipp}
\theta_j (s) := \sum_{k=1}^{2^j} f(y_{j,k}) \ind_{\left(y_{j,k-1};y_{j,k}\right]}(s)\,, \quad\mbox{for all $s\in\R$.}
\end{equation}
Let now show that 
\begin{equation}\label{conv:ftothetaj}
\lim_{j\rightarrow +\infty} \| f-\theta_j \|_{\La(\R)}=0.
\end{equation}
Using the definition of the $\La(\R)$-norm, (\ref{eq1:prop:ipp}), (\ref{eq3:prop:ipp}), the mean value theorem and (\ref{eq2:prop:ipp}), one obtains that
\begin{align*}
\| f-\theta_j \|_{\La(\R)}^{\al} & := \int_{\R} \big|f(s)-\theta_j(s)\big|^{\al} ds = \int_{\R} \Big| \sum_{k=1}^{2^j} (f(s)-f(y_{j,k})) \ind_{\left(y_{j,k-1};y_{j,k}\right]}(s) \Big|^{\al} ds \\
& = \sum_{k=1}^{2^j} \int_{\R} \big|f(s)-f(y_{j,k})\big|^{\al} \,\ind_{\left(y_{j,k-1};y_{j,k}\right]}(s) ds\\
&  \le \Big(\sup_{s\in [a;b]} |f'(s)|^{\al} \Big) \sum_{k=1}^{2^j} \int_{\R} (y_k-s)^{\al}  \ind_{\left(y_{j,k-1};y_{j,k}\right]}(s) ds \\
& =\Big(\sup_{s\in [a;b]} |f'(s)|^{\al} \Big)(\al+1)^{-1}\,(b-a)^{\al+1} \sum_{k=1}^{2^j} 2^{-j(\al+1)}\\
& = \Big(\sup_{s\in [a;b]} |f'(s)|^{\al} \Big)(\al+1)^{-1}\,(b-a)^{\al+1}\, 2^{-j\al}\, , 
\end{align*}
which shows that (\ref{conv:ftothetaj}) is satisfied. Then, one can deduce from (\ref{conv:ftothetaj}) and from a classical property of an $\al$-stable stochastic integral, with $\al>1$, (see \citep{SamTaq}) that 
\begin{equation*}
\lim_{j\rightarrow +\infty}\ESP\bigg |\int_{\R} f(s) \Za{ds} - \int_{\R} \theta_j(s) \Za{ds}\bigg|=0\,.
\end{equation*}

On the other hand, using (\ref{eq3:prop:ipp}), standard computations and elementary properties of a stable stochastic integral, one gets that
\begin{align*}
& \int_{\R} \theta_j(s) \Za{ds} = \sum_{k=1}^{2^j} f(y_{j,k}) \Big( \Za{y_{j,k}} - \Za{y_{j,k-1}} \Big) \\
& = f(y_{j,2^j})\Za{y_{j,2^j}}-f(y_{j,1})\Za{y_{j,0}} - \sum_{k=1}^{2^j-1} \Big(f(y_{j,k+1})-f(y_{j,k})\Big) \Za{y_{j,k}}.
\end{align*}
Notice that the inclusion (\ref{eq1:prop:ipp}) implies that $f(y_{j,2^j})\Za{y_{j,2^j}}= f(b)\Za{b}=0$ and that $\ESP \big |f(y_{j,1})\Za{y_{j,0}}\big|= \ESP \big |f\big(a+2^{-j}(b-a)\big)\Za{a}\big| \rightarrow 0$, when $j\rightarrow +\infty$. Also notice that it follows from Taylor formula that 
\begin{align*}
& \sum_{k=1}^{2^j-1} \Big(f(y_{j,k+1})-f(y_{j,k})\Big) \Za{y_{j,k}} \\
& = \sum_{k=1}^{2^j -1 } (y_{j,k+1}-y_{j,k})f'(y_{j,k}) \Za{y_{j,k}} \\
& \qquad+ \sum_{k=1}^{2^j-1} (y_{j,k+1}-y_{j,k})^2\,\Za{y_{j,k}} \int_{0}^{1} (1-s) f''\big(y_{j,k}+s(y_{j,k+1}-y_{j,k})\big) ds.
\end{align*}
Let us now show that the absolute first moment of the last sum converges to $0$ when $j\rightarrow +\infty$. Using the triangle inequality, the inequalities 
$\sup_{x\in [a;b]} |f''(x)|<\infty$ and $\sup_{x\in [a;b]} \ESP|\Za{x}|<\infty$, and the equality (\ref{eq2:prop:ipp}), one obtains that 
\begin{align*}
& \ESP\bigg |\sum_{k=1}^{2^j-1} (y_{j,k+1}-y_{j,k})^2\,\Za{y_{j,k}} \int_{0}^{1} (1-s) f''\big(y_{j,k}+s(y_{j,k+1}-y_{j,k})\big) ds \bigg| \\
&  \le  \Big (\sup_{x\in [a;b]} |f''(x)| \Big) \Big(\sup_{x\in [a;b]} \ESP|\Za{x}| \Big) \sum_{k=1}^{2^j-1} (y_{j,k+1}-y_{j,k})^2 \\
& \le (b-a)^2\Big (\sup_{x\in [a;b]} |f''(x)| \Big) \Big(\sup_{x\in [a;b]} \ESP|\Za{x}| \Big)  2^{-j} \rightarrow 0\,,\,\, \mbox{ when } j\rightarrow +\infty.
\end{align*} 
In order to complete our proof, it remains to show that
\begin{equation}\label{conv:somRiem}
\lim_{j\rightarrow +\infty}\ESP \bigg|\sum_{k=1}^{2^j -1 } (y_{j,k+1}-y_{j,k})f'(y_{j,k}) \Za{y_{j,k}} -\int_{a}^{b} f'(s)\Za{s} ds\bigg|=0\,.
\end{equation}
It follows from the equality $f'(y_{j,0})=f'(a)=0$ and from elementary properties of a stable stochastic integral that
\[
\sum_{k=1}^{2^j -1 } (y_{j,k+1}-y_{j,k})f'(y_{j,k}) \Za{y_{j,k}}=\int_{a}^{b} \bigg(\sum_{k=0}^{2^j-1} f'(y_{j,k}) \Za{y_{j,k}} \ind_{(y_{j,k};y_{j,k+1}]}(s)\bigg) ds\,.
\]
Thus, using the Fubini-Tonelli theorem and (\ref{eq2:prop:ipp}), it turns out that for deriving (\ref{conv:somRiem}) it is enough to show that
\begin{equation}
\label{conv:somRiem1}
\lim_{j\rightarrow +\infty} 2^{-j}\sum_{k=0}^{2^j-1}\,\sup_{s\in [y_{j,k};y_{j,k+1}]} \Big\{\ESP\big |f'(y_{j,k}) \Za{y_{j,k}}-f'(s) \Za{s}\big |\Big\}=0\,.
\end{equation}
It results from properties of $f'$ and $\mathrm{Z}_{\al}$ that there exists a finite constant $c$ such that, for all $(x,y)\in [a;b]^2$, one has
\[
\big | f'(x)-f'(y)\big|\le c|x-y|\quad\mbox{and} \quad \ESP\big | \Za{x}-\Za{y}\big|\le c|x-y|^{1/\al}\,.
\]
Finally, using these two inequalities, the equality (\ref{eq2:prop:ipp}), and standard computations one can get (\ref{conv:somRiem1}).
\end{proof}

\begin{proof}[Proof of Proposition~\ref{omega1}]
First observe that, it follows from (\ref{daubdef:ejk}), Lemma~\ref{prop:ipp} and the change of variable $u=2^js-k$, that one has, almost surely,
\begin{equation}
\label{eq5:omega1}
\epsilon_{j,k} := 2^{j/\alpha}\int_{\R}  \psi(2^js-k) Z_{\alpha}(ds) = -2^{j/\al} \int_{\R} \psi'(u) \Za{(u+k)2^{-j}} du.
\end{equation}
Next, let $R$ be a fixed positive real number such that $\supp{\psi}\subseteq [-R;R]$. Then using (\ref{eq5:omega1}) and the assumption that $|k|2^{-j}\le l$, one gets, almost surely, that 
$$
|\epsilon_{j,k}|\le C' 2^{j/\al}\,,
$$
where $C'$ is the positive and finite random variable, not depending on $(j,k)$, defined, almost surely, as
\[
C':=\bigg (\sup_{|x|\le R+l} \left|\Za{x} \right| \bigg)\times\left(\int_{-R}^{R} \left|\psi'(u)\right| du \right)\,.
\]
Notice that the almost sure finiteness of $C'$ is mainly a consequence of the fact that the sample paths of the L\'evy $\sas$ process $\{\Za{s}:s\in\R\}$ are, with probability~$1$, c\`adl\`ag functions on $\R$.
\end{proof}


\begin{proof}[Proof of Proposition \ref{prop:regul2}]
First observe that in order to show that (\ref{eq1:regul2}) is satisfied, it is enough to prove that 
\begin{eqnarray}
\label{eq2:regul2}
&& \sup\bigg\{\frac{|\ddot{X}(u_1,v,\omega)-\ddot{X}(u_2,v,\omega)|}{|u_1-u_2|^{v-1/\alpha}} : (u_1,u_2,v)\in [-M;M]^2\times [a;b] \\
&&\hspace{9cm} \mbox{ and } |u_1-u_2|>1 \bigg\}<\infty\nonumber
\end{eqnarray}
and
\begin{eqnarray}
\label{maj:sups1}
&& \sup\bigg\{\frac{|\ddot{X}(u_1,v,\omega)-\ddot{X}(u_2,v,\omega)|}{|u_1-u_2|^{v-1/\alpha}} : (u_1,u_2,v)\in [-M;M]^2\times [a;b] \\
&&\hspace{8.5cm} \mbox{ and } 0<|u_1-u_2|\le 1 \bigg\}<\infty\nonumber
\end{eqnarray}
Notice that Remark~\ref{rem:cont-lowhigh} in Section~\ref{sect:metho} easily implies that (\ref{eq2:regul2}) is satisfied. So from now on, we focus 
on  (\ref{maj:sups1}). 

Let $(u_1,u_2,v)\in [-M;M]^2 \times [a;b]$ be arbitrary but such that $0<|u_1-u_2|\le 1$. In view of Relation~(\ref{eq:localisation}), there exists a constant $c>0$, which does not depend on $u_1,u_2$ and $v$, such that, for all $(j,k) \in \Z^+ \times \Z$, one has
\begin{equation}\label{maj:Psi1}
\left| \Psi(2^ju_1-k,v)-\Psi(2^j u_2-k,v) \right| \le c_1 \Big ( (3+|2^ju_1-k|)^{-2} + (3+|2^ju_2-k|)^{-2} \Big)\,.
\end{equation}
Moreover, 
$$
\left| \Psi(2^ju_1-k,v)-\Psi(2^j u_2-k,v) \right|
$$
can be bounded more sharply when the condition
\begin{equation}\label{maj:uu12}
2^j|u_1-u_2| \le 1
\end{equation}
holds. More precisely, using the mean value theorem and (\ref{eq:localisation}), one has
\begin{align}\label{maj:Psi2}
& \left| \Psi(2^ju_1-k,v)-\Psi(2^j u_2-k,v) \right| \nonumber \\
& \le 2^j |u_1-u_2| \sup_{(u,v)\in [u_1\wedge u_2,u_1\vee u_2] \times[a,b]} | (\partial_x \Psi)(2^ju-k,v)| \nonumber\\
&  \le c 2^j |u_1-u_2| \sup_{u\in [u_1\wedge u_2,u_1\vee u_2] } \left(3+|2^ju-k| \right)^{-2} \le c 2^j |u_1-u_2|  (2+|2^ju_1-k|)^{-2}.
\end{align}
Next, one denotes by $j_0$ the greater integer such that 
\begin{equation}\label{def:j0}
2^{-1}< 2^{j_0} |u_1-u_2| \le 1.
\end{equation}
Observe that one necessarily has that $j_0\ge 0$ since $|u_1-u_2| \le 1$. By combining together (\ref{maj:Psi1}), (\ref{maj:Psi2}) and (\ref{def:j0}) together, one gets
\begin{align}\label{maj:decompAj0Bj0}
& \sum_{j\ge 0} 2^{-jv} \sum_{k\in\Z} |\epsilon_{j,k}| \left| \Psi(2^ju_1-k,v)-\Psi(2^j u_2-k,v) \right| \nonumber \\
& \le c\left( A_{j_0}(u_1,v) 2^{-j_0} + B_{j_0}(u_1,u_2,v) \right),
\end{align}
where
\begin{equation}\label{def:Aj0}
A_{j_0}(u_1,v) = \sum_{j= 0}^{j_0} 2^{j(1-v)} \sum_{k\in\Z} |\epsilon_{j,k}| (2+|2^ju_1-k|)^{-2}
\end{equation}
and
\begin{equation}\label{def:Bj0}
B_{j_0}(u_1,u_2,v) = \sum_{j=j_0+1}^{+\infty} 2^{-jv} \sum_{k\in\Z} |\epsilon_{j,k}| \Big ( (3+|2^ju_1-k|)^{-2} + (3+|2^ju_2-k|)^{-2}  \Big)\,.
\end{equation}
In order to provide appropriate upper bounds for $A_{j_0}(u_1,v) 2^{-j_0}$ and $B_{j_0}(u_1,u_2,v)$, one needs to introduce, for all fixed integer $j\ge 0$, the two sets of indices $k$, $D_j^1$ and $D_j^2$, defined as
\begin{equation}\label{def:DJ12}
D_j^1:=\Big\{ k\in\Z\,:\,  |k|2^{-j}\le 2(M+1)  \Big\} \mbox{ and } D_j^2:=\Big\{ k\in\Z\,:\,  |k|2^{-j} > 2(M+1)  \Big\}\,.
\end{equation}
Notice that $\Z=D_j^1 \cup D_j^2$ and $D_j^1 \cap D_j^2 =\emptyset$.

Let us first provide an appropriate upper bound for $A_{j_0}(u_1,v) 2^{-j_0}$. On one hand
Relation~(\ref{eq1:omega0}) and standard computations give us 
\begin{align}
& \sum_{j= 0}^{j_0} 2^{j(1-v)} \sum_{k\in D_j^2} |\epsilon_{j,k}| (2+|2^ju_1-k|)^{-2} \nonumber \\
& \le C \sum_{j= 0}^{j_0} 2^{j(1-v)} (3+j)^{1/\al+\eta} \sum_{k\in D_j^2} \frac{(3+|k|)^{1/\al+\eta}}{(2+|2^ju_1-k|)^{2}} \nonumber \\
& \le C \sum_{j= 0}^{j_0} 2^{j(1-v)} (3+j)^{1/\al+\eta} \sum_{k\in D_j^2} (1+|k|)^{-(2-1/\al-\eta)} \nonumber \\
& \le C \sum_{j= 0}^{j_0} 2^{j(1-v)} (3+j)^{1/\al+\eta}\,  2^{-j(1-1/\al-\eta)} = C \sum_{j= 0}^{j_0} 2^{-j(v-1/\al-\eta)} (3+j)^{1/\al+\eta} \nonumber \\
& \le C \sum_{j= 0}^{j_0} 2^{-j(a-1/\al-\eta)} (3+j)^{1/\al+\eta} \le C \sum_{j= 0}^{+\infty} 2^{-j(a-1/\al-\eta)} (3+j)^{1/\al+\eta} = C_1 <\infty\,; \label{maj:Aj0:Dj2}
\end{align}
notice that in the previous inequalities, and in the rest of this proof, $C$ denotes the same positive and finite random variable as in~(\ref{eq1:omega0}). 
On the other hand, using Proposition~\ref{omega1}, one has
\begin{align}
& \sum_{j= 0}^{j_0} 2^{j(1-v)} \sum_{k\in D_j^1} |\epsilon_{j,k}| (2+|2^ju_1-k|)^{-2}  \le C'\sum_{j=0}^{j_0} 2^{j(1-v+1/\al)} \sum_{k\in D_j^1} (2+|2^ju_1-k|)^{-2} \nonumber \\
& \le C' \sup_{x\in [0,1]} \left( \sum_{k\in\Z} (1+|x-k|)^{-2} \right) \sum_{j=0}^{j_0} 2^{j(1-v+1/\al)} \le C_2 2^{j_0(1-v+1/\al)}.  \label{maj:Aj0:Dj1}
\end{align}
notice that in the previous inequalities $C'$, and in the rest of this proof, denotes the same positive and finite random variable as in~(\ref{eq1:omega1}). 
Next, it follows from (\ref{def:Aj0}), (\ref{maj:Aj0:Dj2}), (\ref{maj:Aj0:Dj1}) and (\ref{def:j0}) that
\begin{align}
& A_{j_0}(u_1,v) 2^{-j_0} = \left( \sum_{j= 0}^{j_0} 2^{j(1-v)} \sum_{k\in D_j^1} |\epsilon_{j,k}| (2+|2^ju_1-k|)^{-2} \right. \nonumber \\
& \qquad \qquad \qquad + \left. \sum_{j= 0}^{j_0} 2^{j(1-v)} \sum_{k\in D_j^2} |\epsilon_{j,k}| (2+|2^ju_1-k|)^{-2}\right) 2^{-j_0} \nonumber \\
& \le \left(C_1+ C_2 2^{j_0(1-v+1/\al)} \right) 2^{-j_0}   \le C_3 |u_1-u_2|^{v-1/\al}. \label{maj:Aj0}
\end{align}

Let us now provide an appropriate upper bound for $B_{j_0}(u_1,u_2,v)$. On one hand, using Relation~(\ref{eq1:omega0}) and standard computations one gets
\begin{align}\label{maj:Bj0:Dj2}
& \sum_{j=j_0+1}^{+\infty} 2^{-jv} \sum_{k\in D_j^2} |\epsilon_{j,k}| \left( (3+|2^ju_1-k|)^{-2} + (3+|2^ju_2-k|)^{-2} \right)\nonumber \\
& \le C \sum_{j=j_0+1}^{+\infty} 2^{-jv} (3+|j|)^{1/\al+\eta} \sum_{k\in D_j^2} \left( \frac{(3+|k|)^{1/\al+\eta}}{(3+|2^ju_1-k|)^{2}} + \frac{(3+|k|)^{1/\al+\eta}}{(3+|2^ju_2-k|)^{2}} \right) \nonumber \\
& \le C \sum_{j=j_0+1}^{+\infty} 2^{-jv} (3+|j|)^{1/\al+\eta} \sum_{k\in D_j^2} (1+|k|)^{-(2-1/\al-\eta)}. \nonumber\\
& \le C \sum_{j=j_0+1}^{+\infty} 2^{-j(1+v-1/\al-\eta)} (3+|j|)^{1/\al+\eta}  \le C 2^{-(j_0+1)(1+v-1/\al-\eta)} (2+j_0)^{1/\al+\eta} \nonumber \\
& \le C 2^{-(j_0+1)(v-1/\al)} \sup_{n\in\mathbb{N}}\Big\{2^{-n(1-\eta)}(1+n)^{1/\al+\eta}\Big\}\le C_4 |u_1-u_2|^{v-1/\al}. 
\end{align}
On the other hand, using Proposition~\ref{omega1}, one has 
\begin{align}\label{maj:Bj0:Dj1}
& \sum_{j=j_0+1}^{+\infty} 2^{-jv} \sum_{k\in D_j^1} |\epsilon_{j,k}| \left( (3+|2^ju_1-k|)^{-2} + (3+|2^ju_2-k|)^{-2}     \right)\nonumber \\
& \le C'\sum_{j=j_0+1}^{+\infty} 2^{-j(v-1/\al)} \sum_{k\in D_j^1} \left( (3+|2^ju_1-k|)^{-2} + (3+|2^ju_2-k|)^{-2} \right)\nonumber \\
& \le C'\left( 2 \sup_{x\in [0,1]} \sum_{k\in\Z} (3+|x-k|)^{-2} \right) \sum_{j=j_0+1}^{+\infty} 2^{-j(v-1/\al)} \nonumber \\
& \le C_5 2^{-(j_0+1)(v-1/\al)} \le C_5 |u_1-u_2|^{v-1/\al}.
\end{align}
Next combining (\ref{maj:Bj0:Dj2}) and (\ref{maj:Bj0:Dj1}), one obtains that
\begin{equation}\label{maj:Bj0}
B_{j_0}(u_1,u_2,v) \le C_6 |u_1-u_2|^{v-1/\alpha},
\end{equation}
where $C_6$ is a positive and finite random variable. Finally, Relations (\ref{maj:decompAj0Bj0}), (\ref{maj:Aj0}) and (\ref{maj:Bj0}) allow us to derive (\ref{maj:sups1}). 
\end{proof}

\begin{proof}[Proof of Proposition \ref{prop:regul1}]
Let $\omega\in\Omega_0^*$ be arbitrary and fixed. In view of Lemma~\ref{localisation} and of Remark~\ref{rem:cont-lowhigh} in Section~\ref{sect:metho}, for proving the proposition, it is enough to show that, for all fixed integer $r\ge 1$ and for any fixed $(p,q)\in \{1,2,3\}\times \Z^+$, the two series of real numbers
\[
\sum_{j=1}^{+\infty} \sum_{k\in\Z}  \epsilon_{-j,k}(\o)2^{jv} \big (( \partial_v^r \Psi)(2^{-j} u-k,v)-( \partial_v^r \Psi)(-k,v)\big)
\]
and
\[
\sum_{j=1}^{+\infty} \sum_{k\in\Z}   \epsilon_{-j,k}(\omega) 2^{-j(p-v)} (\partial_u^p \partial_v^q \Psi)(2^{-j} u-k,v)
\]
are uniformly convergent when $(u,v) \in [-M;M]\times [a;b]$, where $M,a,b$ arbitrary and fixed real numbers satisfying $M>0$ and $1/\alpha<a < b<1$. In fact it is sufficient to show that these two series are normally convergent, in sense of the uniform (semi)-norm on the compact rectangle $[-M;M]\times [a;b]$. That is one has, for all fixed integer $r\ge 1$,
\begin{equation}\label{dpu:Xpt:oub}
\sum_{j=1}^{+\infty} \sum_{k\in\Z} \big |\epsilon_{-j,k}(\o)\big|\sup_{(u,v)\in  [-M;M]\times [a;b]}\Big\{2^{jv} \big |( \partial_v^r \Psi)(2^{-j} u-k,v)-( \partial_v^r \Psi)(-k,v)\big|\Big\}<\infty\,,
\end{equation}
and, one has, for any fixed $(p,q)\in \{1,2,3\}\times \Z^+$,
\begin{equation}\label{dpu:Xpt}
\sum_{j=1}^{+\infty} \sum_{k\in\Z}   \big|\epsilon_{-j,k}(\omega)\big|\sup_{(u,v)\in  [-M;M]\times [a;b]} \Big\{ 2^{-j(p-v)} \big |(\partial_u^p \partial_v^q \Psi)(2^{-j} u-k,v)\big|\Big\}<\infty.
\end{equation}
Next, observe that a sufficient condition for having (\ref{dpu:Xpt:oub}) is that (\ref{dpu:Xpt}) holds when $(p,q)=(1,r)$. Indeed, it follows from the mean value theorem that
\begin{align*}
&\sup_{(u,v)\in  [-M;M]\times [a;b]}\Big\{2^{jv} \big |( \partial_v^r \Psi)(2^{-j} u-k,v)-( \partial_v^r \Psi)(-k,v)\big|\Big\} \\
& \hspace{3.7cm}\le M\times\sup_{(u,v)\in  [-M;M]\times [a;b]} \Big\{ 2^{-j(1-v)} \big |(\partial_u^1 \partial_v^r \Psi)(2^{-j} u-k,v)\big|\Big\}\,.
\end{align*}
So, from now on, our goal is to derive (\ref{dpu:Xpt}). One denotes by $[\,\cdot\,]$ the integer part function and by $\{\,\cdot\,\}$ the fractional part function. One assumes that $\eta$ is an arbitrarily small fixed positive real number. By using Relation~(\ref{eq1:omega0}), Relation~(\ref{eq:localisation}) and the triangle inequality, one has, for all $(u,v)\in  [-M;M]\times [a;b]$, 
\begin{align}
& \sum_{j=1}^{+\infty}  \sum_{k\in\Z}   \left| \epsilon_{-j,k}(\omega) \right| \times   2^{-j(p-v)}\left| (\partial_u^p \partial_v^q \Psi)(2^{-j} u-k,v) \right| \nonumber \\
& \le C(\omega) \sum_{j=1}^{+\infty} 2^{-j(p-v)} (3+j)^{1/\al+\eta} \sum_{k \in \Z} \frac{(3+|k|)^{1/\al+\eta}}{(3+|2^{-j}u-k|)^2} \nonumber \\
& = C(\omega) \sum_{j=1}^{+\infty} 2^{-j(p-v)} (3+j)^{1/\al+\eta} \sum_{k \in \Z} \frac{(3+|k+[2^{-j}u]|)^{1/\al+\eta}}{(3+|2^{-j}u - [2^{-j}u] -k|)^2} \nonumber \\
& \le C(\omega) \sum_{j=1}^{+\infty} 2^{-j(p-b)} (3+j)^{1/\al+\eta} \sum_{k \in \Z} \frac{(4+|k|+M|)^{1/\al+\eta}}{(3+|\{2^{-j}u\} -k|)^2}\,. \nonumber 
\end{align}
Notice that the last inequality follows from the inequalities $v\le b$ and $\big |[2^{-j}u]\big|\le 1+2^{-j}|u|\le 1+M$. Also notice that $C(\o)$ is a finite constant not depending on $(u,v)$. Next, let $\kappa$ be the constant defined as:
$$
\kappa:= \sup_{x\in [0;1]} \left\{\sum_{k \in \Z} \frac{(4+|k|+M|)^{1/\al+\eta}}{(3+|x -k|)^2}\right\}.
$$
From the inequalities $2-1/\al-\eta>1$ and $3+|x -k| \ge 2+|k|$, for all $(k,x)\in\Z\times [0,1]$, one deduces that $\kappa$ is finite. Then, using our previous calculations, one obtains that
\begin{align}
& \sum_{j=1}^{+\infty} \sum_{k\in\Z}   \big|\epsilon_{-j,k}(\omega)\big|\sup_{(u,v)\in  [-M;M]\times [a;b]} \Big\{ 2^{-j(p-v)} \big |(\partial_u^p \partial_v^q \Psi)(2^{-j} u-k,v)\big|\Big\}\nonumber\\
& \le \kappa C(\omega)  \sum_{j=1}^{+\infty} 2^{-j(p-b)} (3+j)^{1/\al+\eta} <\infty\,, \nonumber
\end{align}
where the last inequality results from the inequalities $p\ge 1>b$. This proves that (\ref{dpu:Xpt}) holds. 
\end{proof}

\section*{Acknowledgements}
The authors are very grateful to the editor and to the anonymous associate editor and two referees for their valuable comments and suggestions which have led to great improvements of the article. This work has been partially supported by ANR-11-BS01-0011 (AMATIS), GDR 3475 (Analyse Multifractale), and ANR-11-LABX-0007-01 (CEMPI).

\end{document}